\documentclass[11pt]{article}

\usepackage{latexsym,amsfonts,amsmath,amssymb,amsthm}
\usepackage[T1]{fontenc}
\usepackage{graphics,graphicx}
\usepackage{color}
\usepackage{subcaption}

\usepackage[font={footnotesize}]{caption}

\usepackage{url}

\newtheorem{lem}{Lemma}[section]
\newtheorem{prop}[lem]{Proposition}
\newtheorem{rem}[lem]{Remark}

\newtheorem{theo}[lem]{Theorem}

\newtheorem*{alpatience}{Heap sorting algorithm for $(U_i,\nu_i)$}

\renewcommand{\P}{\mathbb{P}}
\newcommand{\R}{\mathbb{R}}
\newcommand{\N}{\mathbb{N}}

\newcommand{\E}{\mathbb{E}}

\newcommand{\defeq}{:=}

\newcommand{\dr}{\mathbf{R}}

\begin{document}

\title{Almost-sure asymptotic for the number of heaps inside a random sequence}
\author{\textsc{Basdevant A.-L.}\footnote{Laboratoire Modal'X, Université Paris Nanterre, France. email: anne-laure.basdevant@u-paris10.fr} and \textsc{Singh A.}\footnote{Laboratoire de Mathématiques d'Orsay, Univ. Paris-Sud, CNRS, France. email: arvind.singh@math.u-psud.fr}}
\date{\today}   
\maketitle

\begin{abstract}
We study the minimum number of heaps required to sort a random sequence using a generalization of Istrate and Bonchis's algorithm (2015). In a previous paper, the authors proved that the expected number of heaps grows logarithmically. In this note, we improve on the previous result by establishing the almost-sure and $L^1$ convergence.
\end{abstract}

\bigskip

\noindent{\bf {\textsc MSC 2010 Classification}:} 60F15,  60G55, 60K35.\\
\noindent{\bf Keywords:} Hammersley's process; Heap sorting, Patience sorting; Longest increasing subsequences; Interacting particles systems; Almost-sure convergence.

\section{Introduction}
The so-called \emph{Ulam's problem} consists in estimating the length of the longest decreasing subsequence in a uniform random permutation $\sigma$ of $\{1,\ldots,n\}$. By duality, this question is equivalent to computing the minimal number of disjoint increasing sub-sequences of $\sigma$ required to partition $\{1,\ldots,n\}$. In \cite{Byersetal}, Byers \emph{et al} proposed variations on this problem where the question of finding \emph{increasing subsequences} in a permutation is replaced by that of finding \emph{heapable subsequences}. Subsequently, Istrate and Bonchis \cite{IstrateBonchis}  introduced a modification of the classical patience sorting algorithm called \emph{heap sorting algorithm} which now computes the minimal number of binary heaps required to partition $\{1,\ldots,n\}$. 
\medskip

In \cite{BGGS}, we study a generalization of the algorithm which also allows for the heaps to be random. More precisely, let $\mu$ be a fixed offspring distribution on $\{1,2,\ldots,\}$. Let $(U_i,\nu_i)$ be an i.i.d. sequence where $U_i$ and $\nu_i$ are independent, $U_i$ is uniform on $[0,1]$ and  $\nu_i$ is distributed as $\mu$. We use the following streaming algorithm to sort this sequence into Galton-Watson heaps \emph{i.e.} labeled Galton-Watson trees with the condition that the label of each vertex is larger than that its ancestors. 

\begin{alpatience}\
\begin{itemize}
\item We start at time $1$ with a single tree containing a unique vertex $(U_1,\nu_1)$ and set $\mathbf{R}(1) =1$.
\item At time $n$, we have $\dr(n)$ trees.  To each vertex of these trees is associated a pair 
$(U,\nu)$. The variable $U$ represents the label of the vertex whereas $\nu$ prescribes the maximum number of offsprings that the vertex may have. A vertex $(U,\nu)$ is said to be \emph{alive} if it  has strictly less than $\nu$ children. 
\item At time $n+1$, we add $(U_{n+1},\nu_{n+1})$ as the children of the vertex which is still alive and which has the largest label smaller than $U_{n+1}$. If no such vertex exists, we create a new tree with root $(U_{n+1},\nu_{n+1})$. 
\end{itemize}
\end{alpatience}

This algorithm sorts the sequence $(U_i,\nu_i)$, in their order of arrival, and in such way that
\begin{enumerate}
\item All the trees have the heap property.
\item The trees are asymptotically Galton-Watson distributed with offspring distribution $\mu$ and, at all time, the vertex with label $U_i$ has at most $\nu_i$ children.
\end{enumerate}
\begin{figure}
\begin{center}
\includegraphics[width=14.5cm]{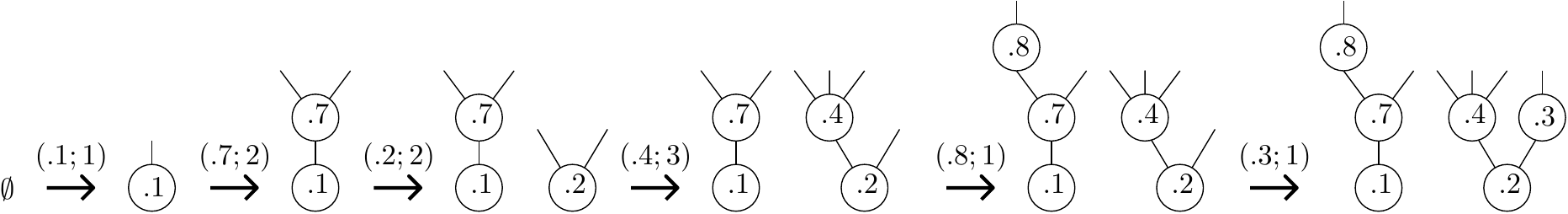}
\caption{ \label{fig:exampleHSA} Example of heap sorting algorithm for the sequence $(.1,1),(.7,2),(.2,2),(.4,3),(.8,1),(.3,1), \ldots$ }
\end{center}
\end{figure}
See Figure \ref{fig:exampleHSA} for an illustration of the procedure. It is easy to verify that, remarkably, this greedy algorithm is optimal in minimizing the number of trees at all time. In \cite{BGGS}, we proved
 that, for any offspring distribution $\mu$ which is not the Dirac mass in $1$ (\emph{i.e.} we exclude Ulam's problem), then the expectation of the number of trees grows logarithmically as it was predicted in \cite{IstrateBonchis}:
 \begin{equation}
\hbox{there exists $c_\mu \in (1,\infty)$ s.t.}\quad\lim_{n\to \infty} \frac{\E[\mathbf{R}(n)]}{\log n} = c_\mu.
 \end{equation}
 The aim of this note is to bootstrap the result above, proving that the limit of $\dr(n)/\log n$ also holds almost surely and in $L^1$.
\begin{theo}\label{theo:main}
For any offspring distribution $\mu \neq \delta_1$, there exists $c_\mu \in (1,\infty)$ such that 
\begin{equation*}
 \lim_{n\to \infty} \frac{\mathbf{R}(n)}{\log n} = c_\mu \qquad \mbox{a.s. and in $L^1$.}
 \end{equation*}
 \end{theo}
 
As explained in \cite{BGGS} (and briefly recalled in the next section), we can associate to the heap sorting algorithm a particle system which plays the same role as \emph{Hammersley's line particle system} for Ulam's problem. One of the main result of \cite{BGGS} states that this particle system, while initially defined on compact intervals, can be extended to an infinite particle system on the whole line. Thus, the strategy to prove Theorem \ref{theo:main} is to first establish the almost sure convergence for an analog of $\dr(n)$ associated with this infinite system on $\R$ and then transfer the result back to the discrete case. In this study, the key ingredients are the remarkable scaling properties of the infinite volume system together with monotonicity arguments.

\section{Almost-sure convergence for the process on the half-plane}
We start by recalling the construction of the \emph{Hammersley's tree process} associated with the heap sorting algorithm introduced in \cite{BGGS}. Let $\mathbb{H}$ denote the upper half-plane $\R\times (0,\infty)$. Consider a point Poisson process (PPP) 
$$\Xi = (U_i,T_i,\nu_i)$$
 on $\mathbb{H}\times \N$ with intensity $du \times dt \times \mu$. For any $ a<b$, we consider the following particle system $H$ on $(a,b)\times \N$ constructed from the atoms of $\Xi$ inside the strip $(a,b)\times (0,\infty)$.
 \begin{itemize}
 \item There is no particle at time $t=0$.
 \item Given $H(t^-)$, an atom $(u,t,\nu)$ of $\Xi$ with $u\in (a,b)$ creates in $H(t)$ a new particle at position $u$ with $\nu$ lives. Furthermore, the particle in $H(t^-)$ with the largest label smaller than $u$ loses
one life (if such a particle exists) and is removed from the system if it was its last life. 
 \end{itemize}  
We can represent the genealogy of the particles using a set of vertical and horizontal lines. Here, vertical lines denote the positions of particles through time and horizontal lines connect particles to their father on their left (or to the vertical axis if they have no father). We denote $\mathcal{G}_{a,b}$ this graphical representation of the process. See Figure \ref{fig:graphrepresentation} for an illustration. 
\begin{figure}
\begin{center}
\includegraphics[height=5.5cm]{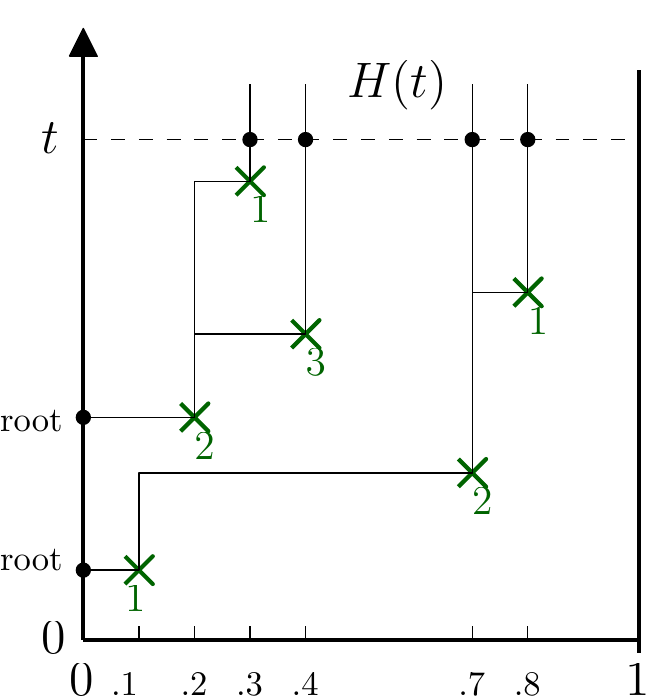}
\caption{ \label{fig:graphrepresentation} An example of the graphical representation $\mathcal{G}_{0,1}$ (which is, in fact, an embedding of the sequence of Figure \ref{fig:exampleHSA}). Crosses represent the atoms of $\Xi$. At time $t$, $H(t)$ has four particles located at position $0.3,\; 0.4,\; 0.7,\; 0.8$ with respective number of lives $1,\; 3, \; 1, \; 1$. } 
\end{center}
\end{figure}
\medskip

For $a=0$ and $b=1$, this particle system may be seen as a continuous time embedding of the heap sorting algorithm where new labels now arrive with Poissonian rate instead of integer time. Therefore, the heaps created by the algorithm are exactly the trees ``drawn'' by the graphical representation. In particular, the number of trees (equiv. heaps) created between time $s$ and $t$  is equal to the number of horizontal lines in $\mathcal{G}_{0,1}$ intersecting the vertical segment $\{0\}\times [s,t]$.
\medskip

Since incoming particles do not affect particles already present on their right, it is clear that the graphical representations $\mathcal{G}_{a,b}$ are compatible for different values of the left boundary \emph{i.e.}
$$
\hbox{for $a' < a$ the restriction of $\mathcal{G}_{a',b}$ to $(a,b)\times (0,\infty)$ coincides with $\mathcal{G}_{a,b}$.}
$$
Thus, there is no problem to define $\mathcal{G}_{-\infty,b}$. Clearly, this compatibility relation does not hold anymore when it is the right boundary that extends since new particle may ``kill'' their left neighbour. However, Theorem 4.4 of \cite{BGGS} states that the graphical representation $\mathcal{G}_{-\infty,b}$ still converges locally, almost surely,  as $b$ tends to infinity, to a random graphical representation $\mathcal{G}_\infty$ on $\mathbb{H}$. This limiting graphical representation is such that there is only finitely many horizontal and vertical lines crossing any compact set inside $\mathbb{H}$. On the other hand, there is an accumulation of horizontal lines at the bottom of the half plane \emph{i.e.} near the X-axis. See Figure \ref{Fig:binaryHalfPlane} for a picture showing how this graphical representation $\mathcal{G}_\infty$ looks like. 
\medskip
\begin{figure}
\includegraphics[width=16cm]{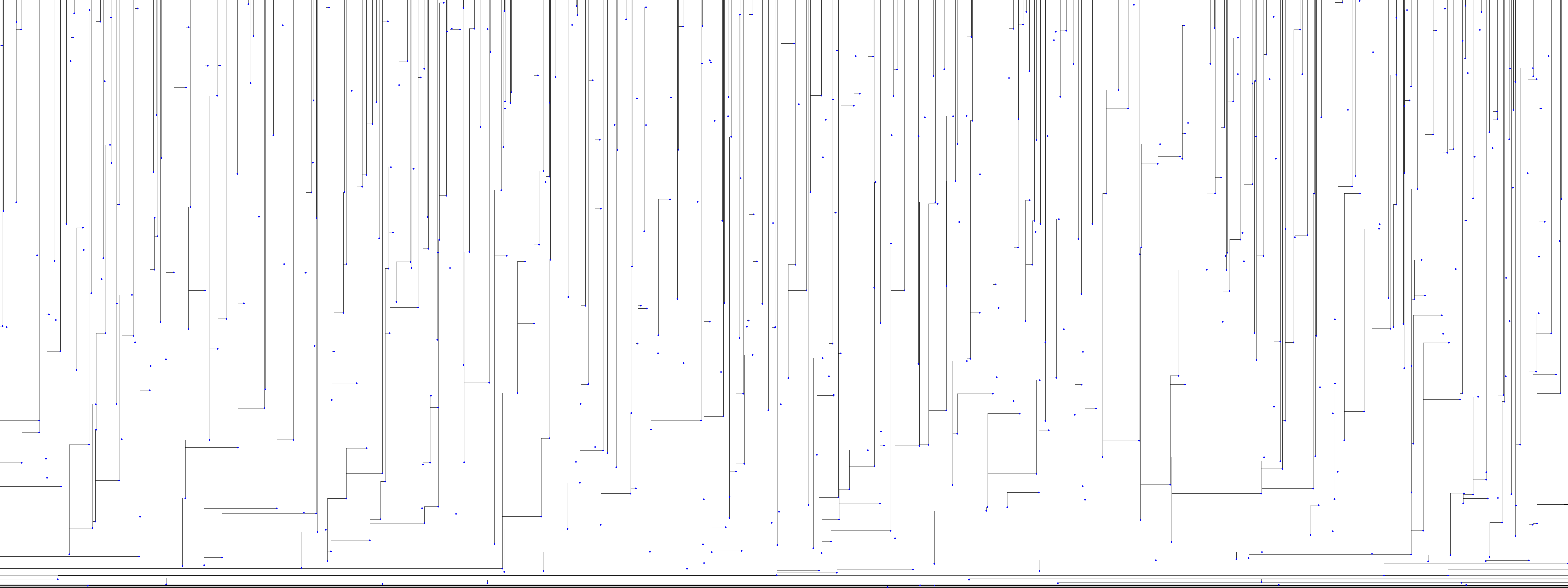}
\caption{\label{Fig:binaryHalfPlane} Simulation of the full half plane representation $\mathcal{G}_\infty$ in the case of binary heaps ($\mu = \delta_2$). The box displayed is $[0,40]\times (0,15]$. There is an accumulation of horizontal lines at $y= 0$ and of vertical lines at $y = +\infty$. }
\end{figure}

The following result is the counterpart of Theorem \ref{theo:main} for the infinite volume system.

\begin{prop} \label{Prop:infini} Let $\mu\neq \delta_1$. Let $\mathcal{G}_\infty$ denote the graphical representation of $H$ on the half plane $\mathbb{H}$. For $0<s<t$, let $R_\infty[s,t]$ be the number of horizontal lines that intersect the segment $\{0\}\times [s,t]$. We have  
$$\lim_{t\to \infty} \frac{R_\infty[1,t]}{\log t} =\E(R_\infty[1,e]) \quad \mbox{ a.s. and in expectation.} $$ 
\end{prop}
Let us point out that this result does not assert the finitness of $\E(R_\infty[1,e])$ (otherwise, the limit above is simply infinite). However, $\E(R_\infty[1,e])$ is indeed always finite as we shall see later.

\begin{proof} We decompose the number of horizontal lines crossing the vertical axis during the time interval $[1,e^n]$ in the following way:
$$R_\infty[1,e^n]=\sum_{i=0}^{n-1} R_\infty[e^i,e^{i+1}].$$
For any $i>0$, the  invariance of the Poisson measure under the mapping 
$$\begin{array}{ccc}
\mathbb{H} & \to& \mathbb{H}\\
(u,t) & \mapsto & (e^i u, \frac{t}{e^i})
\end{array}$$
implies that the law of  $\mathcal{G}_\infty$ is also invariant under this transformation. 
 Thus, it follows that the sequence $(R_\infty[e^i,e^{i+1}],i\ge 0)$ is  stationary. In particular, for any $i\ge 0$, the r.v. $R_\infty[e^i,e^{i+1}]$ has the same law as 
$R_\infty[1,e]$. This already  proves that 
$$\frac{\mathbb{E}[R_\infty[1,e^n]]}{n} =\E(R_\infty[1,e]).$$ 
The sequence $(R_\infty[e^i,e^{i+1}],i\ge 0)$ is clearly not i.i.d. Yet, we will show that it is ergodic since it is mixing. Thus, the ergodic theorem will implies that
\begin{equation}\label{ct1}
\lim_{n\to \infty} \frac{R_\infty[1,e^n]}{n} = \E(R_\infty[1,e])\quad\hbox{a.s.}
\end{equation}
Finally, from \eqref{ct1} and using the monotony of $R_\infty[1,t]$ with respect to $t$, we will conclude that
$$\lim_{t\to \infty} \frac{R_\infty[1,t]}{\log t} =\E(R_\infty[1,e]) \quad \mbox{ a.s.}$$

Thus, it remains to prove that the sequence  $(X_i:=R_\infty[e^i,e^{i+1}],i\ge 0)$ is mixing \emph{i.e.} that for any $n,m$ and any bounded functions $f:\R^{n+1}\mapsto \R$ and $g:\R^{m+1}\mapsto \R$,
 \begin{equation}\label{mix1}
\lim_{k\to \infty}\E\left[f(X_0,\ldots,X_n)g(X_{k},\ldots,X_{k+m})\right]= \E\left[f(X_0,\ldots,X_n)\right]\E\left[g(X_{0},\ldots,X_{m})\right].
\end{equation}
  Fix $n,m\ge 0$ and  $k>n+1$. Let $\bar{X}_k$ denote the number of horizontal lines crossing the segment $\{0\}\times [e^k,e^{k+1}]$ when we remove all the atoms of $\Xi$ below height $e^{n+1}$.
By construction, $\mathcal{G}_\infty\cap (\R\times (0,t))$ is determined by the atoms of $\Xi$ below height $t$. In particular, this implies that $(X_0,\ldots,X_n)$ is independent of $(\bar{X}_k, \ldots, \bar{X}_{k+m})$. Moreover, up to a translation, the graphical representation obtained by removing all atoms below a given height as the same law as $\mathcal{G_\infty}$. Thus, the vector $(\bar{X}_k, \ldots, \bar{X}_{k+m})$ has the same distribution as the vector $(R_\infty[e^k-e^{n+1},e^{k+1}-e^{n+1}],\ldots,R_\infty[e^{k+m}-e^{n+1},e^{k+m+1}-e^{n+1}])$, which is also equal, using the scaling property, to the law of $(R_\infty[1-e^{n+1-k},e-e^{n+1-k}],\ldots, R_\infty[e^{m}-e^{n+1-k},e^{m+1}-e^{n+1-k}])$. Therefore, we obtain the limit in law 
  \begin{equation}\label{lim1}
  \lim_{k\to \infty} (\bar{X}_k,\ldots,\bar{X}_{k+m}) \overset{\mathcal{L}}{=}  (X_0,\ldots,X_m).
  \end{equation}
On the other hand, adding atoms below a given height $s$ can only decrease the number of horizontal lines crossing the segment $\{0\}\times [s,t]$ (see for instance Equation (12) of \cite{BGGS} for more details). This monotonicity result implies that, for any $k>n+1$,
  \begin{equation}\label{mono1}
 X_k \le \bar{X}_k.
  \end{equation}We can now write 
\begin{multline*}
\E\left[f(X_0,\ldots,X_n)g(X_{k},\ldots,X_{k+m})\right]\\
\begin{aligned}
&=\E\left[f(X_0,\ldots,X_n)g(\bar{X}_{k},\ldots,\bar{X}_{k+m})\right]+ \E\left[f(X_0,\ldots,X_n)(g(\bar{X}_{k},\ldots,\bar{X}_{k+m})-g(X_{k},\ldots,X_{k+m}))\right]
 \\
&=\E\left[f(X_0,\ldots,X_n)\right]\E\left[g(\bar{X}_{k},\ldots,\bar{X}_{k+m})\right]
+ \E\left[f(X_0,\ldots,X_n)(g(\bar{X}_{k},\ldots,\bar{X}_{k+m})-g(X_{k},\ldots,X_{k+m}))\right]. &
\end{aligned}
\end{multline*}  
The first term of the r.h.s. of the last equality tends to $\E\left[f(X_0,\ldots,X_n)\right]\E\left[g(X_{0},\ldots,X_{m})\right]$ according to \eqref{lim1}. Concerning the second term, we write 
\begin{eqnarray*}
\E\left[f(X_0,\ldots,X_n)(g(\bar{X}_{k},\ldots,\bar{X}_{k+m})-g(X_{k},\ldots,X_{k+m}))\right]&\le& 2||f||_\infty ||g||_\infty \P\{\exists i \le m, \; \bar{X}_{k+i}\neq X_{k+i}\}\\
  & \le & 2 (m+1) ||f||_\infty ||g||_\infty \sup_{i\ge k} \P\{ \bar{X}_{i}\neq X_{i}\}.
\end{eqnarray*}
Finally, the following easy lemma ascertains that  $\sup_{i\ge k} \P\{ \bar{X}_{i}\neq X_{i}\}$ tends to 0 which concludes the proof of  \eqref{mix1}.

\end{proof}

\begin{lem}
Let $(U_k)$ and $(V_k)$ be two sequences of integer-valued random variables such that
\begin{enumerate}
\item[\textup{(i)}]\ $U_k \leq V_k$ for all $k$.
\item[\textup{(ii)}] The sequence $(U_k)$ is tight.
\item[\textup{(iii)}] $\lim_{k\to\infty} \P\{U_k = a\} - \P\{V_k = a\} = 0$ for every $a$. 
\end{enumerate} 
Then, $$\lim_{k\to\infty}\P\{U_k \neq V_k\} = 0.$$
\end{lem}

\begin{proof}
We first show by induction on $i$ that
$$\lim_{k \to \infty}\P\{V_k\neq i,  U_k= i\}=0.$$
Indeed, we find, using (i), that
\begin{eqnarray}\label{papa}
\P\{V_k\neq 0,  U_k= 0\}&=& \P\{U_k= 0\}- \P\{V_k= 0,  U_k= 0\}\\
&=&\P\{U_k= 0\}- \P\{V_k= 0\},
\end{eqnarray}
which, according to (iii), tends to $0$ as $k$ tends to infinity. Now, for $i\ge 1$, we write
\begin{eqnarray*}
\P\{V_k\neq i,  U_k= i\}&=& \P\{U_k= i \} - \P\{V_k= i,\,  U_k= i\}\\
&=&\P\{U_k= i\} - \P\{V_k= i\}+\P\{V_k= i,\,  U_k<i\}\\
&\le &\P\{U_k= i\} - \P\{V_k= i \}+\sum_{j<i}\P\{V_k\neq j,\,  U_k=j\}.
\end{eqnarray*}
 The induction hypothesis combined with (iii) implies that
the r.h.s. of the last equation tends to 0 as $k$ tends to infinity. Hence, \eqref{papa} holds for all $i$. 
Finally, writing that, for any $A>0$,
$$\P\{V_k\neq U_k\}\le \P\{U_k\ge A\}+ \sum_{i< A} \P\{V_k\neq i,  U_k= i\},$$
and using the tightness of the sequence $(U_k)$,
 we deduce that $\P(V_k\neq U_k)$ tends to 0 as $k$ tends to infinity.
\end{proof}

\section{Almost-sure convergence for the process on $[0,1]$}

We now translate Proposition \ref{Prop:infini} for the Hammersley process defined on the finite interval $[0,1]$.
We use the notation $R_{[a,b]}[s,t]$ for the number of horizontal lines crossing the segment $\{0\}\times [s,t]$ in the graphical representation $\mathcal{G}_{a,b}$ obtained by using only the atoms of $\Xi$ in the strip $[a,b]\times(0,\infty)$.

\begin{prop}\label{Prop:01}  Assume that $\mu\neq \delta_1$. We have
$$\lim_{t\to \infty} \frac{R_{[0,1]}[0,t]}{\log t} =\E\left[R_\infty[1,e]\right] \quad \mbox{ a.s. and in  $L^1$.} $$ 
\end{prop} 
\begin{proof}
As we already noticed, $\mathcal{G}_{0,1}$ coincides with $\mathcal{G}_{-\infty,1}$  restricted to the strip $[0,1]\times (0,\infty)$. Furthermore, taking into account the atoms inside $(1,+\infty) \times (0,\infty)$ can create new horizontal lines inside $[0,1]\times (0,\infty)$ but cannot remove those already present (see Section 2.3.1 of \cite{BGGS} for details). Thus, the horizontal lines of $\mathcal{G}_{0,1}$ are a subset of the horizontal lines of $\mathcal{G}_\infty$.
This domination implies in particular that 
$$R_{[0,1]}[s,t]\le R_\infty[s,t].$$
In particular, we get
$$R_{[0,1]}[0,t]\le R_{[0,1]}[0,1]+R_\infty[1,t]$$
(we need this splitting since $R_\infty[0,1]$ is infinite).
The quantity $R_{[0,1]}[0,1]$ is bounded by the number of atoms in the box $[0,1]^2$.
Thus, in view of Proposition \ref{Prop:infini}, we find that
$$\limsup_{t\to \infty} \frac{R_{[0,1]}[0,t]}{\log t} \le \E(R_\infty[1,e]) \quad \mbox{ a.s. and in expectation.}$$
Let us now prove the matching lower bound. Fix some $N\ge 0$. For $n\ge N$,  we decompose $R_{[0,1]}[0,e^n]$ in the following way
$$R_{[0,1]}[0,e^n]=R_{[0,1]}[0,e^N]+\sum_{i=N}^{n-1} R_{[0,1]}[e^i,e^{i+1}].$$
For $i\ge N$, let $X_i^N\defeq R_{[0,e^{N-i}]}[e^i,e^{i+1}]$ be the number of horizontal lines attached to the Y-axis between heights $e^i$ and $e^{i+1}$ when we consider only the atoms of $\Xi$ with absciss in the interval $[0,e^{N-i}]$. Using the same monotonicity argument as above, we have, for any $i\ge N$,
$$R_{[0,1]}[e^i,e^{i+1}]\ge X_i^N.$$
Thus, for $n\ge N$, we get
$$R_{[0,1]}[0,e^n]\ge R_{[0,1]}[0,e^N]+\sum_{i=N}^{n-1} X_i^N.$$
Using again the invariance of the law of $\Xi$ under the mappings
$$\begin{array}{ccc}
\mathbb{H} & \to& \mathbb{H}\\
(u,t) & \mapsto & ( e^i u, \frac{t}{e^i})
\end{array}$$
we deduce that the sequence $(X_i^N, i\ge N)$ is stationary. In particular, for any $i\ge N$,  $X_i^N$ has the same law as $R_{[0,e^N]}[1,e]$. Again, we prove that the sequence is mixing \emph{i.e.} for any $n,m$ and any bounded functions $f:\R^{n+1}\mapsto \R$ and $g:\R^{m+1}\mapsto \R$,
 \begin{equation*}
\lim_{k\to \infty}\E\left[f(X^N_N,\ldots,X^N_{N+n})g(X^N_{k},\ldots,X^N_{k+m})\right]= \E\left[f(X^N_N,\ldots,X^N_{N+n})\right]\E\left[g(X^N_{N},\ldots,X^N_{N+m})\right].
\end{equation*}
The argument is the same  as in the previous section. Indeed, consider, for $k>N+n$, the number  $\bar{X}^N_k$ of horizontal lines crossing the Y-axis between height $e^k$ and $e^{k+1}$ when we only take into account the atoms of $\Xi$ in the domain $[0,e^{N-k}]\times [e^{N+n+1},\infty)$. It is easily checked that the following holds
\begin{enumerate}
\item $X^N_k \le \bar{X}^N_k$.
\item $(\bar{X}^N_k,\ldots, \bar{X}^N_{k+m})$ is independent of $(X^N_N,\ldots,X_{N+n}^N)$.
\item  $\lim_{k\to \infty} (\bar{X}^N_k,\ldots, \bar{X}^N_{k+m}) \overset{\mathcal{L}}{=}  (X^N_N,\ldots,X_{N+m}^N).$
\end{enumerate}
These three properties imply, just as for Proposition  \ref{Prop:infini}, that the sequence is mixing. Thus, the ergodic theorem implies the almost sure limit
$$\liminf_{n\to \infty} \frac{R_{[0,1]}[0,e^n]}{n}\ge \lim_{n\to \infty} \frac{\sum_{i=N}^{n-1} X_i^N}{n} = \E\left[X_N^N\right]=\E\big[R_{[0,e^N]}[1,e]\big].$$
On the other hand, the sequence $(R_{[0,e^N]}[1,e],  N\ge 0)$  increases to $R_\infty[1,e]$ as $N$ tends to infinity. Thus, the monotone convergence theorem yields
$$\lim_{N \to \infty }\E[X_N^N]=\E\big[R_\infty[1,e]\big]$$
which proves the  convergence of $R_{[0,1]}[0,t]/\log t$ towards $\E\big[R_\infty[1,e]\big]$ almost surely and in expectation.

\medskip

It remains to prove the $L^1$ convergence. Let us first recall that Theorem 2.1 of \cite{BGGS} states that the limit in Proposition \ref{Prop:01}
$$c_\mu\defeq \lim_{t\to \infty} \frac{\E\left[R_{[0,1]}[0,t]\right]}{\log t}$$ 
is finite for any $\mu \neq \delta_1$. We write
\begin{equation*}
\E\left[\left|\frac{R_{[0,1]}[0,t]}{\log t}-c_\mu \right|\right]= \E\left[\left(\frac{R_{[0,1]}[0,t]}{\log t}-c_\mu\right)\right]+ 2\E\left[\left(c_\mu-\frac{R_{[0,1]}[0,t]}{\log t}\right)^+\right].
\end{equation*}
The random variable $\left(c_\mu-\frac{R_{[0,1]}[0,t]}{\log t}\right)^+$ converges a.s. to $0$ and is bounded by $c_\mu$. Thus, the previous convergence in expectation combined with the dominated convergence theorem yield the $L^1$ convergence of $R_{[0,1]}[0,t]/\log t$.

\end{proof}
\begin{rem} In a previous paper \cite{BGGS}, it was shown that the infinite graphical representation exists, which is the same as saying that $R_\infty(s,t)$ is finite for any $0<s<t$. However, it was not proved that the expectation of $R_\infty(s,t)$ is also finite. This is now a consequence of the previous proposition combined with the main result of \cite{BGGS} stating that $c_\mu$ is always finite. Still, we point out that the arguments presented here do not allow, by themselves, to recover that $c_\mu$ is finite.
\end{rem}

We now have  all the tools needed to prove Theorem \ref{theo:main}.  
\begin{proof}[Proof of Theorem \ref{theo:main}]
Recall that the processes $\dr$ and $R_{[0,1]}$ are time changed of each other:
$$\dr(n)=R_{[0,1]}[0,t(n)],$$
where
$$t(n)\defeq \{t\ge 0, \mbox{card}(\Xi\cap[0,1]\times[0,t]\times \N)=n\}$$ counts the number of atoms of $\Xi$ inside the box $[0,1]\times[0,t]$. Since $t(n)/n$ tends a.s. to $1$ as $n$ tends to infinity, we get from Proposition \ref{Prop:01} that
$$\lim_{n\to \infty}\frac{\dr(n)}{\log n}=\lim_{n\to \infty} \frac{R_{[0,1]}[0,t(n)]}{\log t(n)}\frac{\log t(n)}{\log n}= \E\left[R_\infty[1,e]\right] =c_\mu \qquad \mbox{ a.s.}$$
Furthermore, using the convergence in expectation of $\frac{\dr(n)}{\log n}$ towards the same limit (see Theorem 2.1 of \cite{BGGS}), we also deduce the $L^1$ convergence.
\end{proof}


\begin{thebibliography}{99}



\bibitem{BGGS}
A.-L. Basdevant,  L. Gerin, J.-B. Gouéré, and A. Singh. From Hammersley's lines  to Hammersley's trees. Preprint.

\bibitem{Byersetal} J. Byers, B. Heeringa, M. Mitzenmacher, and G. Zervas.  Heapable sequences and subsequences.  \emph{ANALCO11, Workshop on Analytic Algorithmics and Combinatorics} (2011) p.33-44.



\bibitem{IstrateBonchis}
G. Istrate and C. Bonchis. Partition into Heapable Sequences, Heap Tableaux and a Multiset Extension of Hammersley's Process. \emph{Lecture Notes in Computer Science} Combinatorial Pattern Matching (2015) p.261-271.


\end{thebibliography}
\end{document}